\documentclass[11pt,a4paper]{amsart}[2000]
\usepackage{latexsym,amsfonts,amssymb,epsfig,verbatim}
\usepackage{amsmath, latexsym, graphics, textcomp}
\usepackage{amsthm}
\usepackage{pinlabel}
\usepackage{mathtools}
\usepackage{mathrsfs}
\usepackage{xcolor}
\usepackage{enumitem}
\usepackage{tikz}
\usepackage{hyperref}
\usepackage{graphicx}
\usepackage{csquotes}
\usetikzlibrary{decorations.markings}
\usepackage{tikz-cd}

\usepackage[left=3cm,right=3cm, bottom=3cm]{geometry}

\usepackage[all,2cell,dvips]{xy}

\newcommand{\nc}{\newcommand}
\nc{\dmo}{\DeclareMathOperator}
\nc{\nt}{\newtheorem}
\theoremstyle{definition}

\nt{theorem}{Theorem}[section]
\nt{lemma}[theorem]{Lemma}
\nt{prop}[theorem]{Proposition}
\nt{definition}[theorem]{Definition}
\nt{cor}[theorem]{Corollary}
\nt{example}[theorem]{Example}
\nt{remark}[theorem]{Remark}
\nt{notation}[theorem]{Notation}
\nt{fact}[theorem]{Fact}
\nt{question}[theorem]{Question}
\nt{conjecture}[theorem]{Conjecture}

\nc{\Z}{\mathbb{Z}}
\nc{\R}{\mathbb{R}}
\nc{\Q}{\mathbb{Q}}

\nc{\margin}[1]{\marginpar{\tiny #1}}

\nc{\p}[1]{\smallskip\noindent{{\bf #1}}}

\title{Relative profinite rigidity of Baumslag-Solitar groups}

\author{Daxun Wang}
\address{D. Wang: Yau Mathematical Sciences Center, Tsinghua University, Beijing, China}
\email{wangdaxun@mail.tsinghua.edu.cn}



\begin{document}

\begin{abstract}
We prove that each residually finite Baumslag-Solitar group can be distinguished by its finite quotients from all other residually finite Baumslag-Solitar groups. 
\end{abstract}

\maketitle

\section{Introduction}
A pair of finitely generated groups $G$ and $H$ are said to be \textit{profinitely isomorphic} if they share the same isomorphism types of finite quotient groups. It is known that being profinitely isomorphic is equivalent to having isomorphic profinite completions. A finitely generated residually finite group $G$ is said to be (\textit{absolutely}) \textit{profinite rigid} if any finitely generated residually finite group $H$ that is profinitely isomorphic to $G$ is isomorphic to $G$. If we restrict our attention to a class $\mathscr{C}$  of finitely generated residually finite groups, the question of absolute profinite rigidity reduces to a question of \textit{relative profinite rigidity with respect to $\mathscr{C}$}, that is whether any group $H$ in $\mathscr{C}$ that is profinitely isomorphic to $G$ in $\mathscr{C}$ is in fact isomorphic to $G$. We note that residual finiteness is a necessary condition to impose because the free product $G\ast S$ for any finitely generated infinite simple group $S$ is profinitely isomorphic to $G$ but not isomorphic to $G$.

In this note, we restrict our attention to Baumslag-Solitar groups which have presentation as follows:
$$\mbox{BS}(m,n)=\langle a,t\ |\ ta^mt^{-1}=a^n\rangle$$
where $m,n\in \mathbb{Z}_{\neq 0}$. These groups were defined by Baumslag and Solitar in 1962 as examples of non-Hopfian groups \cite{BS62}. Our goal is to study the relative profinite rigidity of Baumslag-Solitar groups. Our method relies on Moldavanski and Sibyakova's paper about finite quotients of one-relator groups \cite{MS95} and Wilkes's paper about profinite rigidity of Seifert fiber spaces \cite{W17}. Our result is as follows.

\begin{theorem}\label{thm: result}
\textit{Among all residually finite Baumslag-Solitar groups, each residually finite Baumslag-Solitar group is relatively profinite rigid.}
\end{theorem}

\section{Relative profinite rigidity}
It follows from the presentation that $\mbox{BS}(m,n)\cong \mbox{BS}(n,m)\cong \mbox{BS}(-m,-n)$, then we can always assume for the rest of the paper that $1\leq m\leq |n|$. It is known that $\mbox{BS}(m,n)$ is residually finite if and only if $m=1$ or $m=\pm n$ with $m>1$ for the last case \cite{M72}.

Let $G$ be a group and $\mathcal{N}$ be the set of all finite index normal subgroups of $G$. We can make $\mathcal{N}$ into direct set by declaring that, for any $M,N\in \mathcal{N}$, $M\leq N$ if $M$ contains $N$. In this case, there are natural epimorphisms $\phi_{NM}:G/N\rightarrow G/M$. We define the \textit{profinite completion} of $G$ as the inverse limit $$\widehat{G}\coloneqq\displaystyle{\lim_{\longleftarrow}}_{N\in \mathcal{N}} G/N$$
If we equip each $G/N, N\in \mathcal{N}$ with the discrete topology, then $\prod_{N\in \mathcal{N}}G/N$ is a compact space and $\widehat{G}$ may be realized as a subspace of $\prod_{N\in \mathcal{N}}G/N$. In addition, there is a natural homomorphism $\iota: G\rightarrow \widehat{G}$ given by $g\mapsto (gN)$. Note that $\iota$ is injective if and only if $G$ is residually finite.

Let $G$ be a finitely generated residually finite group, We denote by $\mathcal{C}(G)$ the set of isomorphism classes of finite quotients of $G$. The following theorem states that $\mathcal{C}(G)$ encodes the same information as $\widehat{G}$. See more details in \cite{DFPR} and Theorem 4.9 of \cite{R13}.

\begin{theorem}\label{thm: finite quotients}
Suppose $G_1$ and $G_2$ are finitely generated residually finite abstract groups, then $\widehat{G}_1\cong \widehat{G}_2$ if and only if $\mathcal{C}(G_1)=\mathcal{C}(G_2)$. 
\end{theorem}

We next turn toward to the proof of Theorem \ref{thm: result}.

\begin{proof}[Proof of Theorem \ref{thm: result}] Let $G\cong \mbox{BS}(m,n)$ and $H\cong \mbox{BS}(m',n')$ be two Baumslag-Solitar groups and assume that $\widehat{G}\cong \widehat{H}$.
Recall that both $\mbox{BS}(m,n)$ and $\mbox{BS}(m',n')$ must be residually finite, we can then divide it into two cases. If $m=1$ or $m'=1$, the Corollary of \cite{MS95} states that any residually finite one-relator group that has the same set of finite quotients as $\mbox{BS}(1,k)$ must be isomorphic to $\mbox{BS}(1,k)$. By Theorem \ref{thm: finite quotients}, this is equivalent to saying $\mbox{BS}(1,k)$ is relatively profinite rigid among all one-relator groups. Since all Baumslag-Solitar groups are one-relator groups, we must have $\mbox{BS}(m,n)\cong \mbox{BS}(m',n')$.

It then remains to consider the case that $m=\pm n$ and $m'=\pm n'$ for some $m,m'>1$. Recall that $\widehat{\mbox{BS}(m,n)}\cong \widehat{\mbox{BS}(m',n')}$. By Proposition 3.2 of \cite{R13} we have $\mbox{BS}(m,n)^{ab}\cong \mbox{BS}(m',n')^{ab}$. Note that the abelianizations of Baumslag-Solitar groups have been computed in Lemma 1 of \cite{BN20}. In particular, we have $\mbox{BS}(m,m)^{ab}\cong \mathbb{Z}^2$ and $\mbox{BS}(m,-m)^{ab}\cong \mathbb{Z}\times \mathbb{Z}_{2m}$. It follows that $\widehat{\mbox{BS}(m,m)}\not\cong \widehat{\mbox{BS}(m',-m')}$ for any $m,m'\in \mathbb{Z}_{\neq 0}$ and $\widehat{\mbox{BS}(m,-m)}\cong \widehat{\mbox{BS}(m',-m')}$ if and only if $m=m'$. Thus it is enough to consider the case where $G\cong \mbox{BS}(m,m)$ and $H\cong \mbox{BS}(m',m')$. 

In this case, it is known that they are fundamental groups of compact Seifert fiber spaces. There is a construction of the Seifert fiber spaces described in \cite{Agol}. In particular, we can take a solid torus, and two parallel annuli in the boundary whose cores run $m$ times around the core of the solid torus. Attaching an annulus$\times \mathbb{I}$ to these two annuli we can get a Seifert fiber space with fundamental group $\mbox{BS}(m,m)$. Now suppose that $\mbox{BS}(m,m)\not\cong \mbox{BS}(m',m')$ (i.e. $m\neq m'$). We denote by $M_m$ (resp. $M_{m'}$) the Seifert fiber space we defined above that has fundamental group isomorphic to $\mbox{BS}(m,m)$ (resp. $\mbox{BS}(m',m')$). Since $m>1$, we have $M_m$ is not a solid torus, $\mathbb{S}^1\times \mathbb{S}^1\times \mathbb{I}$ or the orientable $\mathbb{I}$-bundle over the Klein bottle. This is because the fundamental groups of these three manifolds are $\mathbb{Z}$, $\mbox{BS}(1,1)$ and $\mbox{BS}(1,-1)$. Since $\widehat{\mbox{BS}(m,m)}\cong \widehat{\mbox{BS}(m',m')}$, by Theorem 5.5 of \cite{W17}, $M_m$ and $M_{m'}$ have the same base orbifold. We note that the base orbifold of $M_m$ (resp. $M_{m'}$) is a cone with angle $2\pi/m$ (resp. $2\pi/m'$) and with a ribbon attached to it. So the fundamental group of the base orbifold of $M_m$ (resp. $M_{m'}$) is $\mathbb{Z}\ast \mathbb{Z}_m$ (resp. $\mathbb{Z}\ast \mathbb{Z}_{m'}$). Thus $M_m$ and $M_{m'}$ can not have the same base orbifold unless $m=m'$, which is a contradiction.
\end{proof}

\begin{remark}
    An earlier version of this paper proved that $\mbox{BS}(1,n)$ is relative profinite rigid among all rank 2 residually finite generalized Baumslag-Solitar (GBS) groups for $n$ even. The author thanks an anonymous referee for noting the more general fact that $\mbox{BS}(1,n)$ is relatively profinite rigid among all residually finite GBS groups. The reason is because $\mbox{BS}(1,n)$ are the only solvable residually finite GBS groups \cite{L15}, and it is well known that a finitely generated residually finite group $G$ is solvable of derived length at most $N$ if and only if $\widehat{G}$ is solvable of derived length at most $N$. Therefore, if a residually finite GBS group $H$ has the same profinite completion as $\mbox{BS}(1,n)$, then it must be solvable. This implies that $H$ must be $\mbox{BS}(1,n)$ by comparing the abelianizations. 
\end{remark}

\textbf{Acknowledgements.} The author acknowledges travel support from the Simons Foundation (965204, JM). The author thanks his advisor Johanna Mangahas, and Tamunonye Cheetham-West and Gilbert Levitt for many helpful conversations and for their support.

\end{document}